\documentclass[reqno,11pt]{amsart}
\usepackage{amssymb}
\usepackage{amsmath}
\usepackage{stmaryrd}
\usepackage{graphicx}
\usepackage{color}
\usepackage{cite}
\usepackage[font=small]{caption}
\usepackage{float}
\usepackage{empheq}

\usepackage{tikz}
\usepackage{pgfplots}
\usepackage{wrapfig}
\pgfplotsset{compat=1.10}
\usepgfplotslibrary{fillbetween}
\usetikzlibrary{patterns}
\newtheorem{theorem}{Theorem}[section]

\newtheorem{lemma}[theorem]{Lemma}
\newtheorem{proposition}[theorem]{Proposition}

\newtheorem{remark}[theorem]{Remark}

\definecolor{cadmiumgreen}{rgb}{0.0, 0.42, 0.24}
\numberwithin{equation}{section}
\numberwithin{figure}{section}
\begin{document}

\title[ ]{On a  Free Boundary Model for Three-Dimensional MEMS with a Hinged Top Plate I: Stationary Case}


\author{Katerina Nik}
\address{Leibniz Universit\"at Hannover\\ Institut f\" ur Angewandte Mathematik \\ Welfengarten 1 \\ D--30167 Hannover\\ Germany}
\email{nik@ifam.uni-hannover.de}
%
%
\date{September 27, 2020}


\begin{abstract}
A stationary free boundary problem modeling a three-dimensional electrostatic MEMS 
device is investigated. The device is made of a rigid ground plate and an elastic top plate 
which is hinged at its boundary, the plates being held at different voltages. The model couples a 
nonlocal fourth-order equation for the deformation of the top plate to a Laplace equation for the 
electrostatic potential in the free domain between the two plates. 
The strength of the coupling is tuned by a parameter $\lambda$ which is proportional 
to the square of the applied voltage difference.  
Existence of a stable stationary solution is established  
for small values of $\lambda$. Nonexistence of stationary solutions is obtained  
when $\lambda$ is large enough. 
\end{abstract}
%
\maketitle
%
\section{Introduction}
Microelectromechanical systems (MEMS) are microscopic devices that combine 
electrical and mechanical elements. They often act as sensors or actuators and are used 
in a wide range of nowadays electronics like accelerometers, micropumps, optical switches, 
and microgrippers. Simple idealized MEMS consist of a thin rigid conducting plate above which 
a thin conducting elastic plate is suspended. 
Application of  a voltage difference between the two plates generates a Coulomb force 
which induces a deformation of the elastic plate and thus transforms electrostatic energy
into mechanical energy.
A major factor limiting the effectiveness of such devices is the so-called \textit{pull-in instability} 
occurring when the voltage difference exceeds a certain threshold value. In this case, the Coulomb 
force can no longer be balanced by the response of the elastic plate which then collapses onto the 
rigid plate. To guarantee functionality of such devices it is important to know 
the precise value of this threshold. 

The mathematical description for this type of MEMS involves the deformation of the 
elastic plate and the electrostatic potential between the two plates. 
The governing equation for the elastic plate is derived from its energy balance and involves the 
gradient of the electrostatic potential on the elastic plate. 
The electrostatic potential, in turn, is harmonic in the region enclosed 
by the two plates with given values on both of them. Overall, this leads to a free boundary 
problem, see e.g. 
\cite{EscherLienstrombergsurvey, LaurencotWalker2017, PeleskoBernstein2002} 
and the references therein.
However, most existing literature focuses on two-dimensional MEMS where the elastic 
plate is clamped at its boundary, namely, both position and angle at the boundary are fixed, giving 
rise to Dirichlet boundary conditions for the plate deformation. 

This paper deals with a three-dimensional MEMS device in which the elastic plate is hinged at the 
boundary, meaning that only the position is fixed. This results in a fourth-order equation for the 
elastic plate deformation, but now supplemented with Steklov-type boundary conditions instead of 
Dirichlet ones. Moreover, the right-hand side of this equation, given by the square 
of the gradient of the electrostatic potential on the elastic plate, has much less 
regularity properties due to the fact that the electrostatic potential is now harmonic 
in the three-dimensional region between 
the elastic plate and the rigid plate. We will be more precise later on. 

The paper is organized as follows. The free boundary model is derived in Section \ref{deriv}. 
In Section  \ref{mainresults}, we state our main results regarding existence and nonexistence 
of stationary solutions in dependence on 
the applied voltage difference. We give the corresponding proofs in the subsequent
sections.

\section{Derivation of the Model}
\label{deriv}
In this section, we derive a model for the electrostatic MEMS device depicted in 
Figure \ref{geometry}. 
\begin{figure}
	\centering
	\includegraphics[width=0.76\textwidth]{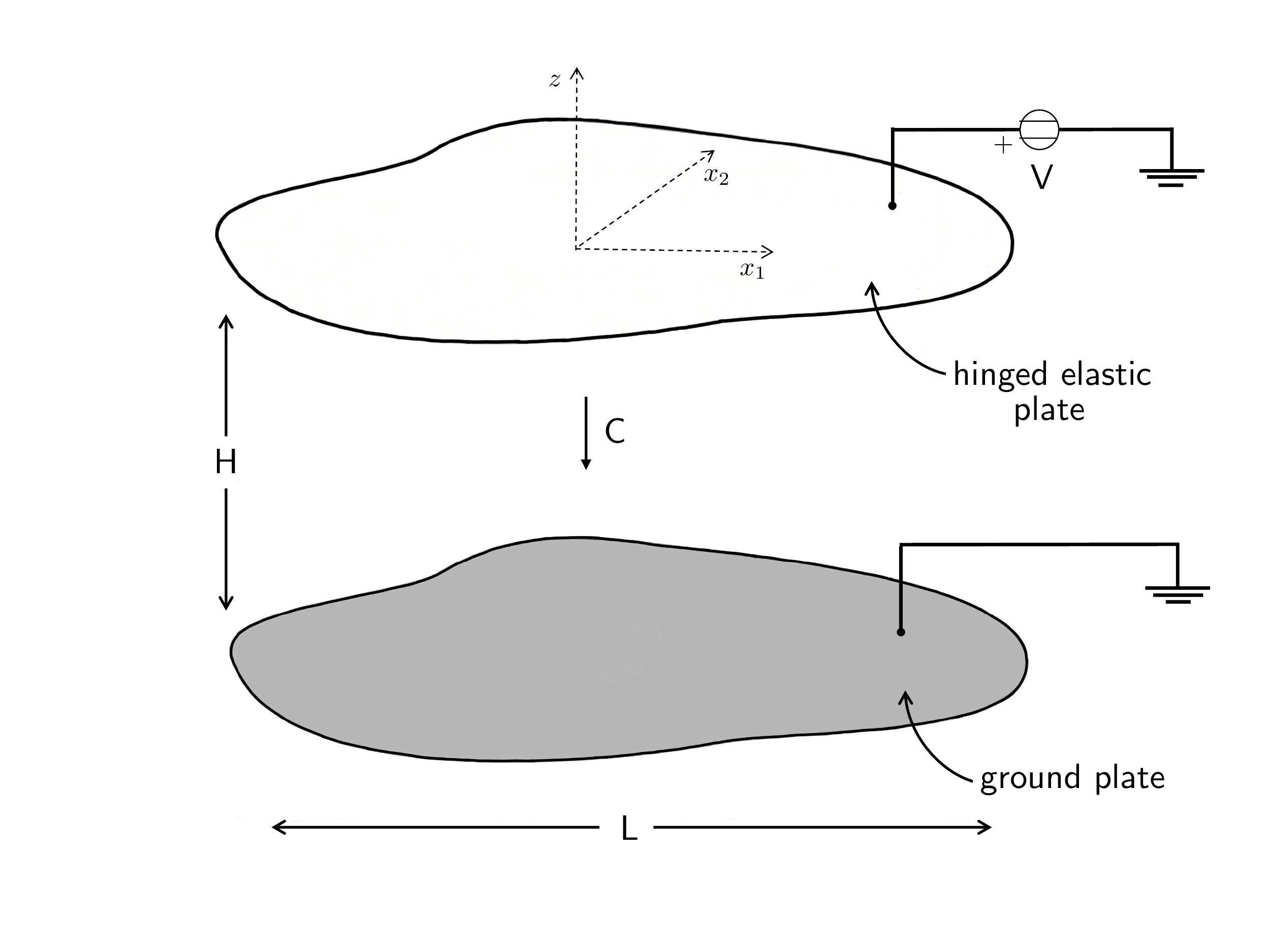}   
	\caption{Idealized electrostatic MEMS device.}
	\label{geometry}
\end{figure}
The device consists of a rigid ground plate of shape  $D\subset \mathbb{R}^2$ above which 
an elastic plate with the same shape $D$ at rest is suspended. 
The ground plate is located at height $z=-H$, while the elastic plate's rest position is at $z=0$. 
We assume that both plates are perfect conductors and are separated by a dielectric medium with 
relative permittivity equal to one. Holding the ground plate at potential zero and applying a 
potential $V>0$ to the elastic plate induce a Coulomb force across the device and thereby a 
deflection of the elastic plate from its rest position, which is assumed to be purely vertical. 
We let $u=u(x)>-H$ denote the vertical deflection 
of the elastic plate from $z=0$ at the point $x=(x_1,x_2)\in D$, and we let $\psi_u=\psi_u(x,z)$ 
denote the electrostatic potential at the point $(x,z)  \in \Omega(u)$, where
\begin{equation*}
\Omega(u):= \{ (x,z) \in D \times \mathbb{R}\, : \, -H<z<u(x)\}
\end{equation*}
is the three-dimensional region between the two plates. We further assume that the elastic plate is 
hinged, that is, we fix the vertical position at the boundary. This gives 
\begin{equation*}
u(x)= 0, \quad x \in \partial D.
\end{equation*}
As we will see later on, the hinged elastic plate also satisfies a natural boundary condition. 
With these assumptions in mind, we formulate the equations governing the electrostatic potential 
in the device and the vertical deflection of the elastic plate. 

\subsection{Governing Equations for $\psi_u$}
\label{gepsi}

The electrostatic potential $\psi_u$ is a solution to the Laplace equation, 
\begin{equation}
\label{modelpsi.1}
\Delta \psi_u =0 \quad \text{in } \: \Omega(u),
\end{equation}
with the boundary conditions 
\begin{equation*}
\psi_u(x,-H) = 0, \quad \psi_u(x,u(x)) =V, \quad x\in D. 
\end{equation*}
We assume that the continuous extension of $\psi_u$ to the vertical sides of $\Omega(u)$ is an 
affine function of $z$, so that
\begin{equation}
\label{modelpsi.2}
\psi_u(x,z)=\frac{V(H+z)}{H+u(x)}\,, \quad (x,z) \in \partial \Omega(u).
\end{equation}
The situation for \eqref{modelpsi.1} and \eqref{modelpsi.2} is depicted in Figure \ref{cross}. 

 \begin{figure}
 	\begin{raggedleft}
 	\begin{tikzpicture}[scale=0.73]
 	\draw[black, line width = 2pt] (-7,0)--(-7,-5);
 	\draw[black, line width = 2pt] (7,-5)--(7,0);
 	\draw[black, line width = 2pt] (-7,-5)--(7,-5);
 	\draw[blue, line width = 2pt] plot[domain=-7:7] (\x,{-1.3-1.7*cos((0.78*pi*\x/7) r)});
 	\draw[black, line width = 1pt, arrows=->] (3,0)--(3,-2.05);
 	\node at (3.65,-0.7) {${\color{black} u(x)}$};
    \node[draw,rectangle,white,fill=white, rounded corners=5pt] at (2,-4.5) {$\Omega_1$};
 	\node at (4.5,-3.5) {${\color{black} \Omega(u)}$};
 	\node at (4,-5.75) {${\color{black} D}$};
	\draw (3.55,-5.75) edge[->,bend left, line width = 1pt] (2.3,-5.1);
	
	\node at (-8.6,0) {$z=0$};
	\node at (-8.4,-5) {$z=-H$};
	\node at (-3.5,-4.6) {$\psi_u=0$};
	\node at (-3.5,-1) {$\psi_u=V$};

    \node at (-0.3,1.7) {$z$};
 	\draw[black, line width = 1pt, arrows = ->,dashed] (0,-5.7)--(0,1.7);
 	\node at (8.6,0.3) {$x_1$};
 	\draw[black, line width = 1pt, arrows = ->,dashed] (-7.7,0)--(8.7,0);
 	\node at (1.53,0.53) {$x_2$};
 	\draw[black, line width = 1pt, arrows = ->,dashed] (0,0)--(1.7,1);
 	
 	\end{tikzpicture}
 	\caption{Cross section of the idealized electrostatic MEMS device.}\label{cross}
 \end{raggedleft}
 \end{figure}
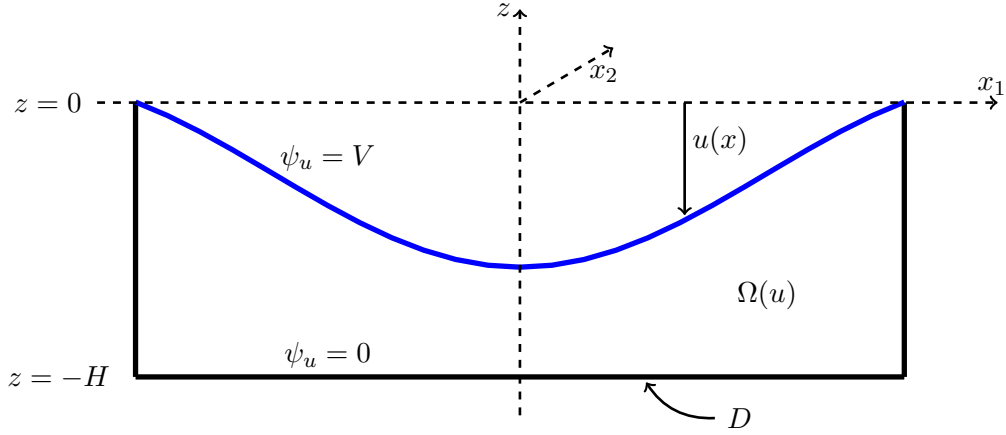

\subsection{Governing Equations for $u$}
\label{geu}

The total energy of the device is constituted of the electrostatic energy 
and the mechanical energy 
and reads 
\begin{equation*}
E(u)= E_e(u)+ E_m(u).
\end{equation*}
The electrostatic energy is given by, see \cite{LaurencotWalkerSIAM2018}, 
\begin{equation*}
\label{electrostatice}
E_e(u) := - \frac{\varepsilon_0}{2} \int_{\Omega(u)} 
\vert \nabla \psi_u \vert^2 \, d(x,z)
\end{equation*}
with $\varepsilon_0$ being the vacuum permittivity. Notice that $E_e(u)$ is nonpositive and that 
$E_e(u)$ depends on $u$ not only through the integral over $\Omega(u)$ but also 
through the solution  $\psi_u$ to \eqref{modelpsi.1}-\eqref{modelpsi.2}. 
The mechanical energy is given by, see \cite{SweersVassi2018}, 
\begin{equation*}
\label{mechanicale}
E_m(u):=B \int_D \Big( \tfrac{1}{2}  (\Delta u)^2 + (1- \sigma) 
\big[(\partial_{x_2}\partial_{x_1}u)^2 - \partial_{x_1}^2 u \,
\partial_{x_2}^2 u\big] \Big)\, dx + \tfrac{T}{2} \int_D  \vert \nabla u \vert^2 \, dx,
\end{equation*}
where the first term accounts for plate bending and torsion with flexural rigidity $B>0$ and 
Poisson ratio $\sigma \in (-1,1)$, and the second term accounts for stretching with stress coefficient 
$T \geq 0$. The second integral in $E_m(u)$ is clearly nonnegative. For the second integral, we 
obtain by applying Young's inequality 
\begin{align*}
&  \int_D \left( \tfrac{1}{2} (\Delta u)^2 + (1- \sigma) 
\big[ (\partial_{x_2}\partial_{x_1}u)^2 - \partial_{x_1}^2 u \, 
\partial_{x_2}^2 u\big] \right) dx 
\\
&\quad =  \int_D \left(\tfrac{1}{2}\, (\partial_{x_1}^2 u)^2 + \tfrac{1}{2}\, (\partial_{x_2}^2 u)^2 
+ (1- \sigma) (\partial_{x_2}\partial_{x_1}u)^2 + \sigma \, \partial_{x_1}^2 u \, 
\partial_{x_2}^2 u \right) dx
\\
& \quad \geq   \int_D \left( \tfrac{1}{2}\, (\partial_{x_1}^2 u)^2 + \tfrac{1}{2}\, (\partial_{x_2}^2 u)^2 
+ (1- \sigma) (\partial_{x_2}\partial_{x_1}u)^2 - \tfrac{\vert \sigma \vert}{2} \,
\big[ (\partial_{x_1}^2 u)^2 + (\partial_{x_2}^2 u)^2  \big]\right) dx
\\
& \quad \geq   \int_D \left( \tfrac{(1- \vert \sigma \vert)}{2} \,
\big[ (\partial_{x_1}^2 u)^2 + (\partial_{x_2}^2 u)^2  \big] 
+ ( 1- \vert \sigma \vert) (\partial_{x_2}\partial_{x_1}u)^2\right) dx
\\
& \quad =  \tfrac{ (1- \vert \sigma\vert)}{2} 
\int_D \left((\partial_{x_1}^2 u)^2 + (\partial_{x_2}^2 u)^2 
+ 2(\partial_{x_2}\partial_{x_1}u)^2 \right)  dx \geq 0.
\end{align*}
So we can make the following remark. 

\begin{remark}
The total energy is the sum of two terms with different signs.
\end{remark}

We now derive the Euler-Lagrange equation and the accompanying natural boundary condition 
by applying a variational principle to the total energy $E(u)$, that is, by finding its critical points. 
We therefore need to compute the first variation of $E(u)$ and to find $u$ such that 
\begin{equation*}
\delta E (u;v) := \frac{d}{dr} E(u+rv)\vert_
{r=0} =0 \quad \text{ for all }  v.
\end{equation*}
Since we suppose the elastic plate to be hinged, an appropriate space in order to look for 
critical points is $\{ v  \in C^{\infty}(\overline{D}) \, : \, v=0 \text{ on } \partial D \}$. 
In the following, we assume that $D$ is a bounded domain in 
$\mathbb{R}^2$ with a sufficiently smooth boundary, such that the exterior unit normal $\nu$ 
and the curvature $\kappa$ of $\partial D$ are well-defined and continuous.

\subsubsection*{Formal derivation of the first variation of the electrostatic energy} 
Here, we follow the same approach as the one discussed in \cite{LaurencotWalkerSIAM2018}.
Let us fix a smooth deformation $u: \overline{D}  \rightarrow \mathbb{R}$ such that 
$u=0$ on $\partial D$ and $u>-H$ in $D$, and let $\psi_u$ be the corresponding solution 
to \eqref{modelpsi.1}-\eqref{modelpsi.2}. Recall that $\psi_u$ 
depends nonlocally on $u$ according to  \eqref{modelpsi.1}. Let $v \in C^{\infty}(\overline{D})$ 
with $v =0$ on $\partial D$ be arbitrary and 
set $u_r:= u + rv$ for $-r_0 <r< r_0$, where $r_0>0$ is chosen sufficiently small so that 
\begin{equation*}
u_r > -H \quad \text{ in } D \; \text{ for all } r \in (-r_0,r_0).
\end{equation*}
Since $u$ is fixed, we write $\psi$ and $\Omega$ rather than $\psi_u$ 
and $\Omega(u)$ in the sequel. In order to compute $\delta \mathcal{E}_e(u;v)$, 
we introduce, for $ r \in (-r_0,r_0)$, the 
transformation $\Phi(r)$ by 
\begin{equation}
\label{trafoP}
\Phi(r)(x,z):= \left(x, z+ rv(x) \frac{H+z}{H+u(x)}\right), \quad (x,z)\in \Omega,
\end{equation}
and notice that 
\begin{equation*}
\Omega(u_r)= \Phi(r)(\Omega) \quad \text{ and } \quad 
\det(\nabla \Phi(r))= 1+ \frac{rv}{H+u} > 0. 
\end{equation*}
Next, let $\psi(r)$ be the solution to \eqref{modelpsi.1}-\eqref{modelpsi.2} in $\Omega(u_r)$, 
that is, 
\begin{empheq}{align}
\Delta \psi(r) &= 0  &&  \hspace{-1.5cm} \text{ in } \Omega(u_r), 
\label{modele2} \\
\psi(r)(x,z) &=V \tfrac{H+z}{H+u_r(x)}, && \hspace{-1.5cm} (x,z) \in \partial \Omega(u_r). 
\label{modele3}
\end{empheq}
Then $\psi(0)= \psi$. Let us now compute the derivative of 
\begin{equation*}
E_e(u_r) = -\frac{\varepsilon_0}{2} \int_{\Omega(u_r)} \vert \nabla \psi(r)
\vert^2 \, d(x,z)
\end{equation*}
with respect to $r$. By the Reynolds transport theorem, see 
\cite[Theorem 5.2.2]{HenrotPierre2018} or 
\cite[XII.Theorem 2.11]{AmannEscher2009}, we deduce that 
\begin{equation*}
\frac{d}{dr} E_e(u_r)\vert_{r=0} = - \varepsilon_0 \int_{\Omega} 
\left( \nabla \psi \cdot \nabla \partial_r \psi(0) + \text{div} \Big(\tfrac{\vert 
\nabla \psi \vert^2}{2} \, \partial_r \Phi(0) \Big)\right) d(x,z).
\end{equation*}
From Gauss' theorem it follows that
\begin{equation*}
\frac{d}{dr} E_e(u_r)\vert_{r=0} = \varepsilon_0 \int_{\Omega} \Delta \psi \, \partial_r \psi(0) \, d(x,z) 
- \varepsilon_0 \int_{\partial \Omega}\Big(\partial_r \psi(0)  \nabla \psi + \tfrac{\vert \nabla 
	\psi \vert^2}{2} \, 
\partial_r \Phi(0) \Big) \cdot n_{\partial \Omega} \, dS,
\end{equation*}
where $n_{\partial \Omega}$ is the outward unit normal of $\partial \Omega$. 
Using that $\Delta \psi = 0$ in $\Omega$ we get 
\begin{equation}
\label{bypartse1}
\frac{d}{dr} E_e(u_r)\vert_{r=0} = -
\varepsilon_0 \int_{\partial \Omega}\Big(\partial_r \psi(0)  \nabla \psi 
+ \tfrac{\vert \nabla \psi \vert^2}{2} \, 
\partial_r \Phi(0) \Big) \cdot n_{\partial \Omega} \, dS.
\end{equation}
By \eqref{trafoP}, we have 
\begin{equation}
\label{trafoP1}
\partial_r \Phi(0)(x,z)= \left(0, v(x)\, \dfrac{H+z}{H+u(x)}\right), \quad (x,z)\in \Omega,
\end{equation}
and thus $\partial_r \Phi(0)=(0,0)$ on $ D \times \{-H\}$. It also holds that 
$\partial_r \psi(0) =0$ on $ D \times \{-H\}$. Since $ v\vert_{\partial D}=0$, $\partial_r \Phi(0)$ and 
$\partial_r \psi(0)$ vanish on $\partial D \times (-H,0)$. Then, \eqref{bypartse1} reduces to 
\begin{equation*}
\frac{d}{dr} E_e(u_r)\vert_{r=0} = -
\varepsilon_0 \int_{\mathfrak{G}}\Big(\partial_r \psi(0)  \nabla \psi 
+ \tfrac{\vert \nabla \psi \vert^2}{2} \, 
\partial_r \Phi(0) \Big) \cdot n_{\mathfrak{G}} \, dS,
\end{equation*}
with
\begin{equation*}
n_{\mathfrak{G}} = \frac{(-\nabla u(x), 1)}{\sqrt{1+ \vert \nabla u(x) \vert^2}}\, , \quad x \in D, 
\end{equation*}
denoting the outward unit normal on the upper boundary $\mathfrak{G}:= \{(x,u(x))\, : \, x \in D \}$ of 
$\Omega$. Hence, 
\begin{align}
\dfrac{d}{dr} E_e(u_r)\vert_{r=0} &= - 
\varepsilon_0 \int_{D}\Big(\partial_r \psi(0)(x,u(x)) \nabla \psi(x,u(x)) 
\nonumber\\[-0.17cm] 
& \hspace{2.8cm} + \frac{\vert \nabla \psi(x,u(x))\vert^2}{2} \, 
\partial_r \Phi(0)(x,u(x)) \Big) \cdot(-\nabla u(x), 1) \, dx.
\label{bypartse2}
\end{align}
If we put $ \nabla':=(\partial_{x_1}, \partial_{x_2})$ and differentiate the boundary 
condition $\psi(x,u(x))=V$ with respect to $x$, we obtain 
\begin{equation}
\label{boundarye1}
\nabla'\psi(x,u(x)) = - \nabla u(x) \partial_z \psi(x,u(x)), \quad x\in D.
\end{equation}
Now recalling from \eqref{modele3} that 
\begin{equation*}
\psi(r)(x,u_r(x)) = V, \quad x\in D,\; r\in (-r_0,r_0),
\end{equation*}
it follows that
\begin{equation*}
\partial_r \psi(0)(x,u(x)) = - \partial_z \psi(x,u(x)) \, v(x), \quad x\in D.
\end{equation*}
This, together with $\partial_r \Phi(0)(x,u(x)) = (0, v(x))$, $x \in D$, and \eqref{boundarye1}, 
then yields that
\begin{equation*}
\frac{d}{dr} E_e(u_r)\vert_{r=0} = \frac{
	\varepsilon_0}{2} \int_{D} \left(1+ \vert \nabla u(x) \vert^2 \right) 
(\partial_z \psi(x,u(x)))^2 \, v(x)\, dx.
\end{equation*}
By \eqref{boundarye1} again, 
\begin{equation*}
\left(1+ \vert \nabla u(x) \vert^2 \right) (\partial_z \psi(x,u(x)))^2 = 
\vert \nabla \psi(x,u(x)) \vert^2, \quad x \in D, 
\end{equation*}
and hence 
\begin{equation}
\label{fvariationelectro3}
\frac{d}{dr} E_e(u_r)\vert_{r=0} = \frac{
	\varepsilon_0}{2} \int_{D} \vert \nabla \psi(x,u(x)) \vert^2 \, v(x)\, dx.
\end{equation}

\begin{remark}
	We note that the above calculations are formal since we did not specify the regularity of 
	$\psi(r)$, neither with respect to $r$ nor with respect to $(x,z) \in \Omega(u_r)$.
\end{remark}

\subsubsection*{Derivation of the first variation of the mechanical energy} 
Fix $ u \in C^{\infty}(\overline{D})$ with $ u =0$ on $\partial D$ and $u>-H$ in $D$. 
Let $v$ and $u_r$ be as above. 
We calculate $\delta E_m(u;v)$ for an arbitrary $v$ as follows. 
First, we have 
\begin{align}
\label{fvariationmechanic}
&\dfrac{d}{dr} E_m(u_r)\vert_
{r=0}
\nonumber \\[-0.11cm] 
&\hspace{0.5cm} = B \int_D \left( \Delta u \, \Delta v + (1-\sigma) \big[ 2 \, 
\partial_{x_2}\partial_{x_1} u 
\, \partial_{x_2}\partial_{x_1} v - \partial_{x_1}^2 u \, \partial_{x_2}^2 v - 
\partial_{x_2}^2 u\,  \partial_{x_1}^2 v\big]\right)  dx
 \\[-0.09cm] 
&\hspace{1cm} + T \int_{D} \nabla u \cdot \nabla v\, dx.
\nonumber 
\end{align}
Performing Green's formula twice and using $ v\vert_{\partial D}=0$, we get for the 
first term in the right-hand side of \eqref{fvariationmechanic}, 
\begin{align}
\label{bypartsm2}
\int_D \Delta u \, \Delta v \, dx 
&= \int_{\partial D}\Delta u \, \partial_{\nu}v \, d\omega   
- \int_D \nabla \Delta u \cdot \nabla v \, dx
\nonumber \\[-0.05cm] 
&= \int_{\partial D}\Delta u \, \partial_{\nu}v \, d\omega 
+ \int_D v \, \Delta^2 u \, dx.
\end{align}
For the second term in the right-hand side, we use Green's 
formula twice and the fact that $u$ vanishes on $\partial D$. We obtain 
\begin{equation*}
2 \int_D  \partial_{x_2}\partial_{x_1} u \, \partial_{x_2}\partial_{x_1} v \, dx 
 = \int_{\partial D } \partial_{x_2}\partial_{x_1} v \left( \nu_2\, \partial_{x_1} u + \nu_1\,
\partial_{x_2} u\right) d\omega + 2 \int_D (\partial_{x_2}^2\partial_{x_1}^2 v) \,  u\, dx  
\end{equation*}
and 
\begin{equation*}
- \int_{D} \partial_{x_1}^2 u \, \partial_{x_2}^2 v \, dx = 
- \int_{\partial D} \nu_1 \, \partial_{x_2}^2 v \, \partial_{x_1} u \, d\omega 
- \int_D (\partial_{x_2}^2\partial_{x_1}^2 v) \,  u\, dx ,
\end{equation*}
\begin{equation*}
- \int_{D} \partial_{x_2}^2 u \, \partial_{x_1}^2 v \, dx = 
- \int_{\partial D} \nu_2 \, \partial_{x_1}^2 v \, \partial_{x_2} u \,  d\omega 
- \int_D (\partial_{x_2}^2\partial_{x_1}^2 v) \,  u\, dx,
\end{equation*} 
with $\nu =(\nu_1,\nu_2)$ denoting the exterior unit normal of $\partial D$. 
Hence
\begin{align*}
&\int_D \left( 2 \, \partial_{x_2}\partial_{x_1} u 
\, \partial_{x_2}\partial_{x_1} v - \partial_{x_1}^2 u \, \partial_{x_2}^2 v - 
\partial_{x_2}^2 u \,\partial_{x_1}^2 v\right) dx
\\[-0.05cm] 
&\quad = \int_{\partial D}\left(\nu_1 \, \partial_{x_2}\partial_{x_1} v \, \partial_{x_2} u
+ \nu_2 \,\partial_{x_2}\partial_{x_1} v \, \partial_{x_1} u 
- \nu_2 \, \partial_{x_1}^2 v \, \partial_{x_2} u  - \nu_1\, \partial_{x_2}^2 v \, \partial_{x_1} u 
\right) d\omega. 
\end{align*}
With the counterclockwise oriented tangent vector $s=(s_1,s_2)$ on $\partial D$ and the fact that
\begin{equation*}
\partial_{x_1} u\vert_{\partial D} = \nu_1 \partial_{\nu} u + s_1 \partial_s u, \quad 
\partial_{x_2} u\vert_{\partial D} = \nu_2 \partial_{\nu} u + s_2 \partial_s u,
\end{equation*} 
and $\partial_s u \vert_{\partial D} =0$ since $u \vert_{\partial D}=0$, we find that
\begin{align*}
&\int_{\partial D}\left(\nu_1\, \partial_{x_2}\partial_{x_1} v \, \partial_{x_2} u 
+ \nu_2 \,\partial_{x_2}\partial_{x_1} v \, \partial_{x_1} u 
- \nu_2 \, \partial_{x_1}^2 v \, \partial_{x_2} u - \nu_1 \partial_{x_2}^2 v \, \partial_{x_1} u \,
\right) d\omega
\\[-0.05cm] 
& \quad = \int_{\partial D}\left( 2 \nu_1 \nu_2 \, \partial_{x_2}\partial_{x_1} v 
- \nu_2^2 \, \partial_{x_1}^2 v - \nu_1^2 \,\partial_{x_2}^2 v \right) \partial_{\nu} u\, d\omega 
\\[-0.05cm] 
& \quad = \int_{\partial D}\left( 2 \nu_1 \nu_2 \, \partial_{x_2}\partial_{x_1} v 
+ \nu_1^2 \, \partial_{x_1}^2 v + \nu_2^2 \,\partial_{x_2}^2 v - \Delta v\right)\partial_{\nu} u\, d\omega
=  \int_{\partial D}\left( \partial_{\nu}^2 v - \Delta v 
\right)\partial_{\nu} u\, d\omega ,
\end{align*}
where in the last step we used 
\begin{equation*}
\partial_{\nu}^2 v = \nu_1^2 \, \partial_{x_1}^2 v 
+ \nu_2^2 \,\partial_{x_2}^2 v + 2 \nu_1 \, \nu_2 \, \partial_{x_2}\partial_{x_1} v. 
\end{equation*}
Using the relation $\Delta v = \partial_{\nu}^2 v + \partial_{s}^2 v + \kappa \, \partial_{\nu} v$, 
see  \cite[§ 4.1]{Sperb1981}, and $\partial_s v = \partial_{s}^2 v =0$ on $\partial D$ since 
$ v\vert_{\partial D}=0$, we get 
\begin{equation*}
\int_{\partial D}\left( \partial_{\nu}^2 v - \Delta v 
\right)\partial_{\nu} u\, d\omega = - \int_{\partial D} \kappa \, \partial_{\nu} v \, 
\partial_{\nu} u\, d\omega.
\end{equation*}
The function $\kappa$ is the signed curvature of $\partial D$.
Consequently, 
\begin{equation}
\label{bypartsm1}
\int_D \left( 2 \, \partial_{x_2}\partial_{x_1} u 
\, \partial_{x_2}\partial_{x_1} v - \partial_{x_1}^2 u \, \partial_{x_2}^2 v - 
\partial_{x_2}^2 u \,\partial_{x_1}^2 v\right) dx
=
- \int_{\partial D} \kappa \, \partial_{\nu} v \, 
\partial_{\nu} u\, d\omega.
\end{equation}
For the last term in the right-hand side of \eqref{fvariationmechanic} we have by Green's formula, 
\begin{equation*}
T \int_D \nabla u \cdot \nabla v \, dx = - T \int_D v \, \Delta u \, dx.
\end{equation*}
From this, together with \eqref{bypartsm2} and \eqref{bypartsm1}, we finally obtain  
\begin{equation}
\label{fvariationmechanic1}
\frac{d}{dr} E_m(u_r)\vert_
{r=0}
= \int_D \left( B \Delta^2 u -T \Delta u \right) v\, dx 
+ B \int_{\partial D} \left( \Delta u - (1-\sigma) \kappa\, \partial_{\nu} u \right) 
\partial_{\nu} v \, d\omega.
\end{equation}

\begin{remark}
	The identity \eqref{bypartsm1} can be proved under weaker 
	regularity assumptions on $u$ and $v$, see \cite{SweersVassi2018}. 
\end{remark}

\subsubsection*{The Euler-Lagrange equation and boundary conditions}

Gathering \eqref{fvariationelectro3} and \eqref{fvariationmechanic1} gives
\begin{align}
\label{fvariationtotal1}
\delta E(u;v) &= \delta\left(E_e(u;v)+ E_m(u;v)\right) \nonumber
\\[-0.05cm] 
&= \int_D \left( B\Delta^2 u -T\Delta u + \tfrac{\varepsilon_0}{2}\,
\vert \nabla \psi_u (x,u(x)) \vert^2 \right)v\, dx 
\nonumber \\[-0.05cm] 
& \quad +
B \int_{\partial D} \left( \Delta u - (1-\sigma) \kappa\, \partial_{\nu} u \right) 
\partial_{\nu} v \, d\omega, 
\end{align}
where $v$ was an arbitrary function in $C^{\infty}(\overline{D})$ such that $v=0$ on $\partial D$.
Setting the first variation equal to zero yields
\begin{equation*}
0= \int_D \left( B\Delta^2 u -T\Delta u + \tfrac{\varepsilon_0}{2}\,
\vert \nabla \psi_u (x,u(x)) \vert^2 \right)v\, dx +
B \int_{\partial D} \left( \Delta u - (1-\sigma) \kappa\, \partial_{\nu} u \right) 
\partial_{\nu} v \, d\omega
\end{equation*}
for all $v \in C^{\infty}(\overline{D})$ with $v=0$ on $\partial D$, and by the 
Fundamental Lemma of Calculus of 
Variations, first in $D$ and then on $\partial D$, it follows that 
\begin{empheq}{align*}
 B \Delta^{2}u -T \Delta u  &= - \tfrac{\varepsilon_0}{2}\, \vert \nabla \psi_u (x,u(x)) \vert^2
\hspace{-2.5cm}  && \text{ in }   D , 
\nonumber   \\[-0.05cm] 
  u&= 0   \hspace{-2.5cm} &&\text{ on } \partial D,
 \\[-0.05cm] 
B (\Delta u -(1 - \sigma) \kappa \partial_{\nu}u ) &= 0 
  \hspace{-2.5cm} && \text{ on } \partial D. 
\nonumber
\end{empheq}
Let us point out that the first boundary condition is an a 
priori boundary condition, whereas the second boundary condition arises as a 
natural boundary condition for $E(u)$. 

\subsection{Governing Equations for $(u,\psi_u)$}
\label{geupsi}
Combining the equations for the electrostatic potential $\psi_u$ and the 
deflection $u$, we arrive at the following equations for an equilibrium state
\begin{empheq}{align}
\Delta \psi_{u} &=0  &&  \hspace{-1.5cm} \text{ in } \Omega(u),
\label{staticpsi.1} \\
\psi_{u}(x,z) &= \tfrac{V(H + z)}{H + u(x)}\, , && \hspace{-1.5cm}  (x,z) \in \partial\Omega(u),
\label{staticpsi.2} \\
B \Delta^{2}u -T \Delta u  &= - \tfrac{\varepsilon_0}{2}\, \vert \nabla \psi_u (x,u(x)) \vert^2 
&&  \hspace{-1.5cm} \text{ in }  D , 
\label{staticu.1}  \\
u= B(\Delta u -(1 - \sigma) \kappa \partial_{\nu}u ) &= 0  
&& \hspace{-1.5cm} \text{ on }  \partial D. 
\label{staticu.2} 
\end{empheq}
Equation \eqref{staticu.1} reflects the balance of mechanical and electrostatic force on 
$\mathfrak{G}_u=\{ (x,u(x)) \\:  x \in D\}$. 
Observe that \eqref{staticpsi.1}-\eqref{staticu.2} is a free boundary problem, since the domain 
$\Omega(u)$ and its boundary component $\mathfrak{G}_u$ 
have to be determined together with the solution $(u,\psi_u)$. 
Furthermore, observe that equation \eqref{staticu.1} 
is a nonlocal fourth-order equation for the deflection $u$, which is coupled to the second-order 
equation in the three-dimensional free domain $\Omega(u)$ for the potential $\psi_u$. 
The Steklov-type boundary conditions \eqref{staticu.2} mean that the elastic plate is hinged. 

We now introduce dimensionless variables in equations \eqref{staticpsi.1}-\eqref{staticpsi.2} for 
$\psi_{u}$ and \eqref{staticu.1}-\eqref{staticu.2} for $u$. We scale $\psi_{u}$ with the applied potential 
$V$, the variable $x$ with a characteristic length $L$ of the device, and both $z$ and $u$ with the size 
of the  gap $H$ between the ground plate and the undeflected elastic plate. So we define 
\begin{equation*}
\tilde{\psi}_{\tilde{u}}:=\frac{\psi_{u}}{V}\,, \quad \tilde{x}:=\frac{x}{L}\,, \quad 
\tilde{z}:=\frac{z}{H}\,, \quad \tilde{u}:=\frac{u}{H}\,, 
\end{equation*}
and the aspect ratio $\varepsilon:=H/L >0$ of the device. Accordingly, we  introduce the sets 
\begin{equation*}
\tilde{D}:=\{\tilde{x}\in \mathbb{R}^2\, : \, L\tilde{x}\in D\}, 
\quad 
\tilde{\Omega}(\tilde{u}):=\{(\tilde{x}, \tilde{z})\in \tilde{D}
\times \mathbb{R}\, : \, -1< \tilde{z} < \tilde{u}(\tilde{x})\} ,
\end{equation*}
and define the parameters 
\begin{equation*}
\beta:= B>0, \quad 
\tau:=PL^2\geq 0,\quad  
\lambda = \lambda(\varepsilon) := \frac{\varepsilon_0 V^2L}{2\varepsilon^3}>0. 
\end{equation*}
We then substitute these relations into \eqref{staticpsi.1}-\eqref{staticu.2} to 
derive dimensionless equations. Dropping the tilde symbol everywhere, we get for the dimensionless 
electrostatic potential 
\begin{empheq}{align}
\varepsilon^2 \Delta' \psi_u + \partial_z^2 \psi_u &=0  &&  \hspace{-2.5cm} \text{ in } \Omega(u),
\label{restatic1} \\[-0.05cm] 
\psi_u(x,z) &= \frac{1+z}{1+u(x)}, && \hspace{-2.5cm} (x,z) \in \partial  \Omega(u),
\label{restatic2}
\end{empheq}
where $\Delta' := \partial_{x_1}^2 +  \partial_{x_1}^2$. Also, we obtain for the dimensionless deflection 
of the elastic plate the equation 
\begin{equation}
\beta \Delta^2 u - \tau \Delta u = - \lambda  \left( \varepsilon^2 \vert \nabla' \psi_u (x,u(x)) \vert^2 + 
(\partial_z \psi_u (x,u(x)))^2 \right), \quad x \in D,
\label{restatic3}
\end{equation}
with $\nabla':=(\partial_{x_1}, \partial_{x_2})$ and subject to the hinged boundary conditions 
\begin{equation}
u=\Delta u - (1 - \sigma) \kappa \partial_{\nu}u  = 0,  \quad x \in  \partial D.
\label{restatic4}
\end{equation}
Equations \eqref{restatic1}-\eqref{restatic4} are the system describing the statics 
of an idealized MEMS with a hinged top plate that we will consider throughout the paper. 
Let us point out that this system 
is only meaningful as long as the elastic plate 
does not touch down on the rigid plate, 
that is, the deflection $u$ satisfies $u>-1$. Indeed, if $u$ reaches the value $-1$ 
somewhere in $D$, the region  $\Omega (u)$ gets disconnected. 
Moreover, the vertical derivative $\partial_z \psi_{u}$ appearing in equation 
\eqref{restatic3} becomes singular at such touchdown points as $\psi_{u}=1$ 
along $z=u(x)$ while $\psi_{u}=0$ along $z=-1$. The singularity is some sense tuned by the 
parameter $\lambda$ which is proportional to the square of the applied voltage and it 
is thus expected that solutions to \eqref{restatic1}-\eqref{restatic4} only exist for small values of 
$\lambda$ below a certain critical threshold. 
Questions related to the pull-in threshold have been a field of very active research in recent years. 
However, most of the research has been focused on a simplified version 
of \eqref{restatic1}-\eqref{restatic4}, which we will shortly introduce here. 

\subsection{Small Gap Model}
\label{sgm}
A common assumption made in most prior works for MEMS is a vanishing aspect ratio 
$\varepsilon = H/L$ that reduces the free boundary problem to a single equation with a right-hand 
side developing a singularity in the moment the elastic plate touches down on the ground plate. 
More precisely, setting formally $\varepsilon =0$ makes it possible to obtain an explicit solution 
\begin{equation*}
\psi_u(x,z) = \frac{1+z}{1+u(x)}, \quad (x,z)\in \overline{\Omega(u)}, 
\end{equation*}
to \eqref{restatic1}-\eqref{restatic2}, where the deformation $u$ now satisfies the 
so-called \textit{small gap model}
\begin{empheq}{align}
\beta \Delta^2 u - \tau \Delta u&=- \frac{\lambda }{(1+u)^2}  &&  \hspace{-2.5cm} \text{ in } D,
\label{sgm1} \\[-0.07cm]
u = \Delta u -(1-\sigma) \kappa \partial_{\nu} u &= 0 && \hspace{-2.5cm} \text{ on } \partial  D.
\label{sgm2}
\end{empheq}
When Navier boundary conditions are considered, that is, when $\sigma =1$ in \eqref{sgm2}, this 
small gap model has been studied extensively, see e.g. the 
monograph \cite{EspositoGhoussoubGuo2010}.  
Roughly speaking, it is known that there is a threshold $\lambda_*>0$ such that there is at 
least one solution for $\lambda \in (0,\lambda_*)$, one solution for $\lambda = \lambda_*$, and 
no solution for $\lambda > \lambda_*$. 
Let us emphasize that the maximum principle is available in this case and turns out to be a 
major tool for the analysis. 
Recently we were able to show a sign-preserving result for the fourth-operator 
$\beta \Delta^2  - \tau \Delta $ under the boundary conditions \eqref{sgm2} in the case $D$ is 
convex with $\partial D \in C^{2,1}$ \cite{SweersVassi2018}. Taking advantage of this property, 
we might also expect a similar result on the pull-in threshold for the free boundary problem 
\eqref{restatic1}-\eqref{restatic4} with $\varepsilon >0$. The next section takes a step in that direction. 

\section{Main Results}
\label{mainresults}
Throughout the paper we assume that 
$D\subset \mathbb{R}^2$ is a bounded convex domain with  $C^4$-boundary,
and that the curvature $\kappa$ is nonnegative on $\partial D$. 

An important ingredient in the analysis of \eqref{restatic1}-\eqref{restatic4} is solving the 
elliptic problem \eqref{restatic1}-\eqref{restatic2} on the three-dimensional nonsmooth domain 
$\Omega(u)$ in dependence of the free boundary described by $u$. Also precise 
information on the trace of $ \nabla \psi_u$ on the elastic plate is required as a function of $u$. 
For this, for a given deflection $u$, one can transform the Laplace equation 
\eqref{restatic1}-\eqref{restatic2} for $\psi_u$ to a boundary value problem with Dirichlet data 
of the form 
\begin{empheq}{align*}
\mathcal{L}_u \phi_u &= 0  &&  \hspace{-3.5cm} \text{ in } \Omega, 
\\[-0.05cm]
\phi_u &=\eta  && \hspace{-3.5cm} \text{ on } \partial  \Omega 
\end{empheq}
for the transformed electrostatic potential 
\begin{equation*}
\phi_u(x,\eta):= \psi_u(x,(1+u(x))\eta -1), \quad (x,\eta) \in \Omega,
\end{equation*}
in the fixed cylinder $\Omega: = D\times (0,1)$.  The differential operator $- \mathcal{L}_u$ has 
coefficients depending on $u$, $\nabla u$ and $\Delta u$ and being singular at touchdown points 
$x\in D$ where $u(x)=-1$. However, for $u>-1$ in $D$, the operator $-\mathcal{L}_u$ is elliptic. 
Employing then elliptic regularity theory and pointwise multiplications in Sobolev spaces, one 
can show the following key result.

\begin{proposition}[\hspace{-0.01cm}{\cite[Proposition 2.1]{LaurencotWalker2016}}]
\label{electrostaticprop}
Define for given $\rho \in (0,1)$ the open subset 
\begin{align*}
S(\rho):= \big \{u \in W^2_3(D) :  \; & u=0 \text{ on } \partial D,  \, 
\Vert u \Vert_{W^2_3(D)} < 1/\rho, \\ &\text{ and } \; u(x) > -1+ \rho \; \text{ for } \, x \in D \big \}
\end{align*}
of $W^2_3(D)$ with closure 
\begin{align*}
\overline{S}(\rho):=\big \{u \in W^2_3(D) :  \; & v=0 \text{ on } \partial D,  \, 
\Vert u \Vert_{W^2_3(D)} \leq 1/\rho, \\&\text{ and } \; u(x) \geq -1+ \rho \; \text{ for } 
\, x \in D \big\}.
\end{align*}
Then, for each $u \in \overline{S}(\rho)$ there is a unique solution $\psi_u \in W^2_2(\Omega(u))$ to 
\eqref{restatic1}-\eqref{restatic2}, and there is a constant $c(\rho)>0$ such that 
\begin{equation}
\label{electrostaticlip}
\Vert \phi_{u_1} - \phi_{u_2} \Vert_{W^2_2(\Omega)} \leq c(\rho) \Vert u_1 - u_2 \Vert_{W^2_3(D)}, 
\quad u_1, u_2 \in \overline{S}(\rho). 
\end{equation}
Moreover, the mapping 
\begin{equation}
\label{rhs}
g: S(\rho) \rightarrow L_2(D), \quad  u \mapsto  \varepsilon^2 \vert \nabla' \psi_u (\cdot,u) \vert^2 + 
(\partial_z \psi_u (\cdot,u))^2
\end{equation}
is analytic, bounded, and globally Lipschitz continuous. 
\end{proposition}
Note that the function $g$ from Proposition \ref{electrostaticprop} is less regular than in the case 
of a two-dimensional domain $\Omega (u)$ studied in \cite{EscherLaurencotWalker2014, 
LaurencotWalker2013, LaurencotWalkerI, LaurencotWalkerII, LaurencotWalkerESAIM}. 

In the companion paper \cite{NikP2020} 
a parabolic evolution model pertaining to the stationary model 
\eqref{restatic1}-\eqref{restatic4} is discussed. The dynamic boundary condition reads 
\begin{align}
\partial_t u + \beta& \Delta^2 u - \tau \Delta u \nonumber
\\[-0.05cm]
&= - \lambda \left( \varepsilon^2 \vert \nabla' \psi_{u(t)} (x,u(t,x)) \vert^2 + (\partial_z \psi_{u(t)}
 (x,u(t,x)))^2 \right), \quad t>0, \, x \in D,  \label{evolutioneq}
\end{align}
which is derived by means of Newton's second law. 
According to Proposition \ref{electrostaticprop}, one can rewrite the evolution problem as a single 
semilinear Cauchy problem for the deflection $u$ that one can then solve using semigroup theory. 
We recall here the local-in-time well-posedness result from \cite{NikP2020} .

\begin{theorem}
	\label{timewellposedness}
Consider an initial value
\begin{equation*}
u^0 \in W^4_2(D) \: \; \text{ such that } \: \; u^0> -1 \:\text{ in } D \:\;\text{ and } \:\; u^0=\Delta u^0-(1-\sigma)\kappa 
\partial_{\nu} u^0=0 \: \text{ on } \partial D.
\end{equation*}
Then, for any value of $\lambda >0$, there is a unique solution $(u(t), \psi_{u(t)})$ to the model with 
the evolution equation \eqref{evolutioneq} on the maximal interval of existence $[0,T_m)$ with 
regularity
\begin{align*}
&u\in C([0,T_m), W^4_2(D)) \cap C^1([0,T_m), L_2(D)), \\
&\psi_{u(t)} \in W^2_2(\Omega(u(t))),
\end{align*}
satisfying 
\begin{equation*}
u(t,x)>-1, \quad (t,x) \in [0,T_m)\times D. 
\end{equation*}
\end{theorem}

The plan of the present paper is first to obtain an existence result for the stationary free boundary 
problem \eqref{restatic1}-\eqref{restatic4}. We prove that for sufficiently small values of  
$\lambda$ there exists a solution $(u,\psi_u)$ to \eqref{restatic1}-\eqref{restatic4} that is 
asymptotically stable. 

\begin{theorem}
	\label{existencestat}
	Let $\rho \in (0,1)$ be fixed. 
	\begin{enumerate}
		\item[(i)]  There are $\lambda_s= \lambda_s(\rho)>0$ and an analytic 
		function $[0, \lambda_s)\rightarrow W^{4}_{2}(D)$, $\lambda 
		\mapsto U_{\lambda}$, such that $(U_{\lambda}, \Psi_{U_{\lambda}})$ 
		is for each $\lambda \in (0, \lambda_s)$ the unique solution 
		of \eqref{restatic1}-\eqref{restatic4} with 
		\begin{equation*}
      	U_{\lambda} \in S(\rho) \: \; \text{ and } \: \; \Psi_{U_{\lambda}} \in W^2_2(\Omega(U_{\lambda})).
		\end{equation*}
		Moreover, $-1< U_{\lambda} \leq 0$ in $D$. \\
		\item[(ii)]  Let $\lambda \in (0, \lambda_s)$. There are numbers $\varpi_0, r_0, R>0$ such that 
		for each initial value $u^0 \in W^{4}_{2}(D)$ satisfying $u^0=\Delta u^0-(1-\sigma)\kappa 
		\partial_{\nu} u^0=0$ on $\partial D$ and 
		\begin{equation*}
	      \Vert u^0- U_{\lambda} \Vert_{W^{4}_{2}(D)} < r_0
		\end{equation*}
        there is a unique global solution $(u(t), \psi_{u(t)})$ of the associated 
        parabolic evolution problem to \eqref{restatic1}-\eqref{restatic4} with
		\begin{align*}
		&u \in C([0,\infty),W^{4}_{2}(D)) \cap C^1([0,\infty),L_2(D)), 
		\\
		&\psi_{u(t)} \in W^2_2(\Omega(u(t))), \quad t \geq 0, 
		\end{align*}
		and $u(t)>-1$ in $D$ for each $t \geq 0$. 
		Moreover, 
		\begin{equation}
		\label{PLS}
		\Vert u(t)- U_{\lambda} \Vert_{W^{4}_{2}(D)} + \Vert \partial_t u(t) \Vert_{L_2(D)}
		\leq R e^{-\varpi_0 t}\Vert u^0- U_{\lambda} \Vert_{W^{4}_{2}(D)}, \quad t \geq 0.
		\end{equation}
	\end{enumerate}
\end{theorem}

The same statement for Dirichlet boundary conditions in place of the hinged ones 
\eqref{restatic4} is proved in \cite[Theorem 1.2]{LaurencotWalker2016}, 
except for the nonpositivity of $U_{\lambda}$. 
This property holds true if $D$ is the unit ball in $\mathbb{R}^2$ so that a sign-preserving 
property in radial symmetry is available according to \cite{LaurencotWalkersign}. 
Note again that in the situation considered herein, 
where $D$ is an arbitrary two-dimensional convex domain and 
where hinged boundary conditions are used, 
a sign-preserving property holds, see \cite{SweersVassi2018}.

Theorem \ref{existencestat} is proven in Section \ref{exsmall} with the help of 
the Implicit Function Theorem 
for part (i) and the Principle of Linearized Stability for part (ii).  Let us remark that Theorem 
\ref{existencestat}  provides a unique stationary solution with first component in $S(\rho)$
for small values of $\lambda$. 
Yet, an open problem is whether there are other solutions for such values of $\lambda$ and what 
one can say about their stability or instability. \\
\\
Our second result shows that there is an upper threshold for $\lambda$ above 
which no solution to \eqref{restatic1}-\eqref{restatic4} exists. 
The proof techniques involved rely on a lower bound of 
$\partial_z \psi_{u}$ on $\mathfrak{G}_u$, see Lemma \ref{nonexistencemax} below, 
and on a positive eigenfunction 
$\varphi_1 \in W^4_2(D)$ associated to a positive eigenvalue $\mu_1>0$ of the operator $\beta 
\Delta^2-\tau \Delta$ subject to the hinged boundary conditions \eqref{restatic4}, which is 
established in the Appendix. 
A similar nonexistence result for large values of $\lambda$ is true when $D$ is the unit ball 
in $\mathbb{R}^2$ and when the hinged boundary conditions 
are replaced by Dirichlet boundary conditions, see \cite[Theorem 1.3]{LaurencotWalker2016}. 
The proof, however, follows a completely different path.

\begin{theorem}
	\label{nonexistencestat}
	Suppose that $\lambda \geq \lambda_{ns}:=\mu_1$. Then there is no solution $(u,\psi_u)$ 
	to \eqref{restatic1}-\eqref{restatic4} with regularity 
	\begin{equation*}
	u \in W^4_2(D), \quad \psi_u \in W^2_2(\Omega(u))
	\end{equation*}
	so that $u(x)>-1$ for $x\in D$. 
\end{theorem}

The proof of Theorem \ref{nonexistencestat} is given in Section \ref{nonexlarge}, but it leaves open the question whether the values $\lambda_s$ and $\lambda_{ns}$ 
coincide or not.

\section{Existence for Small Voltage Values: Proof of Theorem \ref{existencestat}}
\label{exsmall}
We first prove Theorem \ref{existencestat}. We write $W^4_{2,\mathcal{B}}(D)$ 
for the subspace of $W^4_2(D)$ 
consisting of functions 
$u$ satisfying the hinged boundary conditions  \eqref{restatic4}. 
Next we introduce the operator $A \in \mathcal{L}(W^4_{2,\mathcal{B}}(D),L_2(D))$ by 
\begin{equation}
\label{operatorA}
Au:= \beta \Delta^2 u -\tau \Delta u, \quad u \in W^4_{2,\mathcal{B}}(D),
\end{equation}
and recall the following properties from \cite{NikP2020}.

\begin{proposition}
	\label{semigroup}
	The operator $-A$ generates a strongly continuous analytic semigroup on $L_2(D)$ 
	with spectrum contained in $[ \textnormal{Re}z<0 ]$.
\end{proposition}

The proof of this proposition is based on \cite[Theorem 4.1]{Amann1993}, 
which requires the $C^4$-regularity of the boundary $\partial D$. \\
\\
Our proof of Theorem \ref{existencestat} follows similar lines as in \cite[Theorem 1.2]{LaurencotWalker2016}. 
For Theorem \ref{existencestat}(i), we notice that $W^{4}_{2}(D)$ is continuously embedded in $W^{2}_{3}(D)$ and 
recall that $g$ defined in \eqref{rhs} is an analytic map $S(\rho) \rightarrow L_2(D)$.  Thus, since 
$A \in \mathcal{L}(W^4_{2,\mathcal{B}}(D),L_2(D))$ is invertible by Proposition \ref{semigroup}, we 
obtain that the map
\begin{equation*}
F: \mathbb{R}\times (W^{4}_{2,\mathcal{B}}(D) \cap S(\rho)) \rightarrow 
W^{4}_{2,\mathcal{B}}(D), \quad (\lambda, v) \mapsto v + 
\lambda A^{-1}g(v)
\end{equation*}
is analytic with $F(0,0)=0$ and $D_v F(0,0)= \text{id}_{W^{4}_{2,\mathcal{B}}(D)}$. According to the 
Implicit Function Theorem, there is $\lambda_s= \lambda_s(\rho)>0$ and an analytic map 
\begin{equation*}
[\lambda  \mapsto U_{\lambda} ]: [0,\lambda_s)\rightarrow W^{4}_{2,\mathcal{B}}(D)
\end{equation*}
such that $U_0=0$ and $F(\lambda, U_{\lambda})=0$ for $\lambda \in [0,\lambda_s)$. 
For $\lambda \neq 0$, let $\Psi_{U_{\lambda}}$ be the potential 
associated with $U_{\lambda}$. Then $(U_{\lambda}, \Psi_{U_{\lambda}})$ is the unique 
solution to \eqref{restatic1}-\eqref{restatic4} satisfying  
$U_{\lambda} \in W^{4}_{2,\mathcal{B}}(D) \cap S(\rho)$ 
and  $\Psi_{U_{\lambda}} \in W^2_2(\Omega(U_{\lambda}))$. 
The nonpositivity of $U_{\lambda}$ follows from \cite{SweersVassi2018} 
since $-g(U_{\lambda}) \leq 0$. 

We now prove Theorem \ref{existencestat}(ii). 
With the definition of the operator $A$, equations \eqref{evolutioneq} 
and \eqref{restatic4} read 
\begin{equation*}
\partial_t u + Au=-\lambda g(u) .
\end{equation*}
Setting $v=u-U_{\lambda}$, $\lambda \in (0,\lambda_s)$, and 
\begin{equation*}
B_{\lambda}:= \lambda Dg(U_{\lambda}) \in \mathcal{L}(W^{4}_{2,\mathcal{B}}(D), L_2(D)), 
\end{equation*}
we obtain the linearization 
\begin{equation}
\partial_t v + (A+ B_{\lambda}) v =  -\lambda \,\big( g(U_{\lambda}+v)
- g(U_{\lambda}) - Dg(U_{\lambda})v \big),
\label{linearization}
\end{equation}
and denoting the right-hand side of \eqref{linearization} by $G_{\lambda}(v)$, the initial value 
problem
\begin{empheq}{align*}
\partial_t v + (A+ B_{\lambda}) v &= G_{\lambda}(v), \quad t>0, 
 \\[0.1cm]
v(0)&=v^0,
\end{empheq}
where $G_{\lambda} \in C^{\infty}(\mathcal{O}_{\lambda}, L_2(D))$ is 
defined on an open zero neighborhood $\mathcal{O}_{\lambda} \subset W^{4}_{2,\mathcal{B}}(D)$ 
such that $U_{\lambda}+ \mathcal{O}_{\lambda} \subset S(\rho)$. 
Moreover, $G_{\lambda}(0)=0$ and $DG_{\lambda}(0)=0$. Since 
\begin{equation*}
\Vert B_{\lambda}\Vert_{\mathcal{L}(W^4_{2,\mathcal{B}}(D),L_2(D))} \rightarrow 
0 \quad \text{ as } \lambda \rightarrow 0, 
\end{equation*}
it follows from \cite[Proposition I.1.4.2]{Amann1995} that the operator $-(A+B_{\lambda})$ 
is  the generator of a strongly continuous analytic semigroup on $L_2(D)$ with 
a negative spectral bound provided that $\lambda$ is sufficiently small. 
Now we can apply \cite[Theorem 9.1.2]{Lunardi1995} and make $\lambda_s>0$ smaller, 
if necessary, to conclude part (ii) of Theorem \ref{existencestat}. 
\qed
\medskip

From Theorem \ref{existencestat}(ii) and the Lipschitz continuity of $[v\mapsto \phi_v]$ obtained in 
Proposition \ref{electrostaticprop}, we also conclude that 
\begin{equation*}
\Vert \phi_{u(t)} - \phi_{U_{\lambda}}\Vert_{W^2_2(\Omega)} \leq 
R' e^{-\varpi_0 t}\Vert u^0- U_{\lambda} \Vert_{W^{4}_{2}(D)}, \quad t \geq 0, 
\end{equation*}
with a positive constant $R'$.

\section{Nonexistence for Large Voltage Values: Proof of Theorem \ref{nonexistencestat}}
\label{nonexlarge}
We now turn to the proof of Theorem \ref{nonexistencestat}. Consider a solution $(u,\psi_u)$ 
to \eqref{restatic1}-\eqref{restatic4} with $u\in W^4_{2,\mathcal{B}}(D)$, 
$\psi_u \in W^2_2(\Omega(u))$, and $u(x)>-1$ for $x \in D$.  Set for $x \in D$
\begin{equation*}
G(x):= \big(1+\varepsilon^2 \vert \nabla u(x)\vert^2 \big)^2  \left( \partial_z \psi_u(x,u(x)) \right)^2.
\end{equation*}
Since 
\begin{equation*}
\nabla'\psi(x,u(x)) = - \nabla u(x) \partial_z \psi(x,u(x)), \quad x\in D, 
\end{equation*}
by \eqref{restatic2}, 
the function $u$ solves 
\begin{equation}
\label{soluG}
\beta \Delta^2 u - \tau \Delta u=- \lambda G   \quad \text{ in } D
\end{equation}
with hinged boundary conditions \eqref{restatic4}, and we infer from the nonnegativity of $G$ and 
\cite{SweersVassi2018} that 
\begin{equation}
\label{nonpositiveu}
-1< u(x) \leq 0,  \quad x \in  D. 
\end{equation}
To prove Theorem \ref{nonexistencestat} we use the idea from \cite[p.156]{LaurencotWalker2013}. 
The starting point is the following upper bound for 
$\psi_u$.

\begin{lemma}
	\label{nonexistencemax}
	For $(x,z) \in \Omega(u)$, define $M(x,z):=1+z-u(x)$. Then
	\begin{align}
	\label{nonexistencemax1}
	\psi_u(x,z) &\leq M(x,z), \quad (x,z)\in \Omega(u), 
	\\[0.05cm]
	\partial_z \psi_u(x,u(x)) &\geq 1, \quad x\in D.
	\label{nonexistencemax2}
	\end{align}
\end{lemma}

\begin{proof} 
The boundary conditions \eqref{restatic2} and \eqref{restatic4} for $\psi_u$ and $u$ ensure that, 
for $ x\in \partial D$ and $z\in (-1,0)$, 
\begin{equation*}
M(x,z)=1+z=\psi_u(x,z), 
\end{equation*}
while, for $x\in  D$, 
\begin{equation}
\label{nonexistencemax3}
M(x,u(x))=1=\psi_u(x,u(x))
\end{equation}
and 
\begin{equation*}
M(x,-1)=-u(x) \geq 0 =\psi_u(x,-1). 
\end{equation*}
Hence,  $M\geq \psi_u$ on $\partial \Omega(u)$. Moreover, for $(x,z)\in \Omega(u)$,  we have that   
\begin{equation*}
-\varepsilon^2 \Delta'M(x,z)-\partial_z^2M(x,z)= \varepsilon^2 \Delta u(x).
\end{equation*}
Now, in order to verify that $\Delta u\geq0$ in $D$,  we rewrite equations \eqref{soluG} and 
\eqref{restatic4} as the coupled system 
\begin{equation*}
\left.
\begin{array}{rlrl}
	-\beta\Delta v+\tau v \hspace{-0.2cm}&=-\lambda G  &\text{ in }D , 
	\\
	v\hspace{-0.2cm}&=-( 1-\sigma ) \kappa \, 
	\partial_{\nu } u  &\text{ on }\partial D %
\end{array}%
\right. 
\quad \text{ and } \quad 
\left.
\begin{array}{rlrl}
	-\Delta u \hspace{-0.2cm}&=v &\text{in }D , 
	\\
	u\hspace{-0.2cm}&=0 &\text{on }D .%
\end{array}%
\right.  \label{nonexistencemax4}
\end{equation*}
Since $D$ is convex, hence $\kappa \geq 0$, \eqref{nonpositiveu} implies that 
\begin{equation*}
-( 1-\sigma) \kappa \, 
\partial_{\nu } u \leq 0 \text{ on } \partial D,
\end{equation*}
and it follows from $-\lambda G \leq 0$ in $D$ and the maximum principle that $v\leq 0$ in $D$. 
Hence, $\Delta u \geq 0$ in $D$. Then, as 
\begin{equation*}
-\varepsilon^2 \Delta'M(x,z)-\partial_z^2M(x,z) \geq 0 
= -\varepsilon^2 \Delta'\psi_u(x,z)-\partial_z^2 \psi_u(x,z)
\end{equation*}
for $(x,z)\in \Omega(u)$, we can apply the maximum principle to obtain that 
$M \geq \psi_u$ in $\Omega(u)$. This, together with \eqref{nonexistencemax3}, 
yields that, for $x\in D$ and $z\in (-1,u(x))$, 
\begin{equation*}
\frac{\psi_u(x,z)-\psi_u(x,u(x))}{z-u(x)}\geq \frac{M(x,z)-M(x,u(x))}{z-u(x)}=1 .
\end{equation*}
Sending $z$ to $u(x)$, we conclude that $\partial_z \psi_u(x,u(x)) \geq 1$ for all $x\in D$.
\end{proof}

By Lemma \ref{nonexistencemax}, $G(x)\geq 1$ for $x \in D$, so that 
\begin{equation}
\label{inequalityu}
-\beta \Delta^2u +\tau \Delta u \geq \lambda \quad \text{ in } D.
\end{equation}
Next, according to Theorem \ref{eigenvaluetheorem} of the Appendix, 
the operator $\beta \Delta^2-\tau \Delta$ 
with the hinged boundary conditions \eqref{restatic4} has a positive eigenvalue $\mu_1>0$ with a 
corresponding positive eigenfunction $\varphi_1\in W^4_{2,\mathcal{B}}(D)$.  
Multiplying \eqref{inequalityu} by $\varphi_1$ and integrating over $D$ we have 
\begin{equation*}
\lambda \int_D \varphi_1 \, dx \leq \int_D (-\beta \Delta^2u +\tau \Delta u )\varphi_1\, dx.
\end{equation*}
Applying Green's formula and using the boundary condition $\varphi_1=0$ on $\partial D$ yields
\begin{equation*}
\int_D (-\beta \Delta^2u +\tau \Delta u)\varphi_1\, dx
= \beta \int_D \nabla \Delta u \cdot \nabla \varphi_1 \, dx 
- \tau \int_D \nabla u \cdot \nabla \varphi_1\, dx.
\end{equation*}
By using Green's formula twice and taking into account that $u=0$ on $\partial D$ we get
\begin{align*}
&\int_D (-\beta \Delta^2u +\tau \Delta u)\varphi_1\, dx
\\[-0.07cm]
&\quad = \beta \int_{\partial D}
( \Delta u \, \partial_{\nu} \varphi_1 - \Delta \varphi_1 \, \partial_{\nu}u)\,d\omega 
+ \beta \int_D \nabla u \cdot \nabla \Delta \varphi_1 \, dx + \tau \int_D u \, \Delta \varphi_1\, dx.
\end{align*}
Using once more Green's formula and $u=0$ on $\partial D$  we obtain
\begin{align*}
&\int_D (-\beta \Delta^2u +\tau \Delta u)\varphi_1\, dx
\\[-0.07cm]
&\quad = \beta \int_{\partial D}
( \Delta u \, \partial_{\nu} \varphi_1 - \Delta \varphi_1 \, \partial_{\nu}u)\, d\omega 
+ \int_D( -\beta \Delta^2 \varphi_1 + \tau \Delta \varphi_1) u\, dx
\\[-0.07cm]
&\quad =  \int_D ( -\beta \Delta^2 \varphi_1 + \tau \Delta \varphi_1) u\, dx,
\end{align*}
where the last step follows from the second boundary condition for $u$ and $\varphi_1$.  
Finally we end up with 
\begin{equation*}
\lambda \int_D \varphi_1 \, dx \leq 
\int_D ( -\beta \Delta^2 \varphi_1 + \tau \Delta \varphi_1 ) u\, dx 
=  -\mu_1 \int_D \varphi_1 \, u\,  dx < \mu_1 \int_D \varphi_1 \,  dx, 
\end{equation*}
since $u> -1$ in $D$. So $\lambda < \mu_1$, and this completes the proof of Theorem 
\ref{nonexistencestat}. 
\qed


\section{Appendix}

Here we show the existence of a positive eigenfunction for the linear operator 
$\beta \Delta^2-\tau \Delta$  in $W^4_{2,\mathcal{B}}(D)$. 

\begin{theorem}
\label{eigenvaluetheorem}
The eigenvalue problem 
\begin{empheq}{align*}
\beta \Delta^2 \varphi - \tau \Delta \varphi&=\mu \varphi  
&&  \hspace{-3cm} \text{ in } D,
 \\
\varphi= \Delta \varphi -(1-\sigma) \kappa \partial_{\nu} \varphi &= 0 
&& \hspace{-3cm} \text{ on } \partial  D
\end{empheq}
admits a unique eigenvalue $\mu_1$ which has a positive eigenfunction $\varphi_1$. 
The eigenvalue $\mu_1$ is positive and simple. Moreover, $\varphi_1 \in W^{4}_{2,\mathcal{B}}(D)$ 
and $\partial_{\nu} \varphi_1 <0$ on $\partial D$. 
\end{theorem}
The proof of this theorem combines the recent sign-preserving 
result \cite{SweersVassi2018} with the celebrated Kre\mbox{\u\i}n-Rutman theorem.

\begin{proof}
Since  $\partial D \in C^4$, we can apply 
\cite[Theorem 2.20]{GazzolaGrunauSweers2010} (or alternatively, Proposition \ref{semigroup})
to obtain that, for $f\in L_2(D)$, the boundary value problem 
\begin{empheq}{align*}
\beta \Delta^2 \varphi - \tau \Delta \varphi&=f
&&  \hspace{-3cm} \text{ in } D,
\\
\varphi= \Delta \varphi -(1-\sigma) \kappa \partial_{\nu} \varphi &= 0 
&& \hspace{-3cm} \text{ on } \partial  D
\end{empheq}
has a unique solution $\varphi \in W^4_{2,\mathcal{B}}(D)$, which we denote by $\mathcal{S}f$. 
Furthermore, $\mathcal{S}$ belongs to $\mathcal{L}(L_2(D), W^4_{2,\mathcal{B}}(D))$. Setting 
$\mathcal{I}:W^{4}_{2,\mathcal{B}}(D)\rightarrow L_2(D)$ the compact embedding, we find that 
$\mathcal{K}:=\mathcal{S}\mathcal{I}$ is a compact endomorphism of $W^{4}_{2,\mathcal{B}}(D)$. 
We next observe, see \cite{Amann1976, DanersMedina1992}, that $W^{4}_{2,\mathcal{B}}(D)$ 
is an ordered Banach space with positive cone
\begin{equation*}
(W^{4}_{2,\mathcal{B}}(D))_+
:=\{u\in W^{4}_{2,\mathcal{B}}(D)\, : \, u\geq 0 \text{ in } D\}.
\end{equation*}
This cone has a nonempty interior given by 
\begin{equation*}
\text{int}\big( (W^{4}_{2,\mathcal{B}}(D))_+ \big) =
\{u \in W^{4}_{2,\mathcal{B}}(D) \, : \, u >0 \text{ in }D \,
\text{ and } \, \partial_{\nu} u <0 \text{ on } \partial D\}.
\end{equation*}
Now \cite{SweersVassi2018} guarantees that for nonnegative $f \in  W^{4}_{2,\mathcal{B}}(D)$ 
($f \not\equiv 0$), $\mathcal{K}f>0$ in $D$ and $\partial_{\nu} (\mathcal{K}f) <0$  on  $\partial D$. 
So $\mathcal{K}\big( (W^{4}_{2,\mathcal{B}}(D))_+ \setminus \{0\}\big) 
\subset \text{int}\big( (W^{4}_{2,\mathcal{B}}(D))_+ \big) $ and we can 
apply the Kre\mbox{\u\i}n-Rutman theorem, see e.g.  \cite[Theorem 3.2]{Amann1976} or \cite[Theorem 19.3]{Deimling1985}, to complete the proof. 
\end{proof}

\section*{Acknowledgements}
This paper is an edited extract of the author's Ph.D. thesis submitted to the Leibniz Universität Hannover. The author would like to thank Prof. Christoph Walker for his supervision. 
\bibliographystyle{siam}
\bibliography{modstat}

\end{document}